\documentclass[11pt,letterpaper]{amsart}

\usepackage{amsbsy,amssymb,amscd,amsfonts,latexsym,amstext,delarray,amsmath,graphicx,color,caption,blkarray,mathtools,comment}
\usepackage{tikz}
\usetikzlibrary{cd}
\usetikzlibrary{matrix,arrows}
\usepackage{graphicx}
\usepackage{hyperref}
\input xy
\xyoption{all}
\usepackage{euscript}
\usepackage{multirow,charter}
\usepackage{etex, pictexwd,dcpic}
\usepackage[normalem]{ulem}

\usepackage{amscd}
\usepackage{times,fancyhdr}
\usepackage{array}
\usepackage{epsfig}
\usepackage{graphicx}
\usepackage{color}
\usepackage{mathtools}
\usepackage{mathabx}
\usepackage{euscript}
\usepackage{footnote}
\usepackage[normalem]{ulem}

\newcounter{thecounter}
\numberwithin{thecounter}{section}
\newtheorem{lemma}[thecounter]{Lemma}

\newtheorem{theorem}[thecounter]{Theorem}

\theoremstyle{definition}

\newcommand{\Z}{{\mathbb Z}}
\newcommand{\R}{{\mathbb R}}

\begin{document}

\title{Kac--Moody Fibonacci sequences}
\author{Lisa Carbone and Pranav Shankar}
\date{}
\begin{abstract} We summarize known results on how to generate an infinite family of integer sequences from the root lattices of rank 2 Kac--Moody algebras. We compute and tabulate the first twenty entries of a number of these sequences. This provides an overarching framework for a large class of Fibonacci-type integer sequences, evaluations of Chebyshev S and U-polynomials and others.
    \end{abstract}


\maketitle

Let $A$ be the generalized Cartan matrix
$$
A = (a_{ij})_{i,j=1,2} =
\begin{pmatrix}
2 & -a \\
-b & 2
\end{pmatrix}
$$
with $a,b \in \Z_{>0}$, $ab \ge 4$, $a \geq b$, with Kac-Moody
algebra $\mathfrak{g}=\mathfrak{g}(A)$, root system
$\Delta=\Delta(A)$, and Weyl group $W=W(A)$. When $ab=4$, $A$ is
positive semi-definite but not positive definite, and  $A$ is {\it
affine}. When $ab>4$,   $A$ is {\it hyperbolic}.

Let $\Pi = \{\alpha_1, \alpha_2\}$ be the simple roots and
let $w_1, w_2$ be the simple reflections. We have $$W=\langle
w_1\rangle\ast\langle
w_2\rangle\cong\mathbb Z/2\mathbb Z\ast\mathbb Z/2\mathbb Z$$ so
every element of $W$ is of the form
$w_1^{\epsilon_1}(w_2w_1)^nw_2^{\epsilon_2}$,
where $\epsilon_i\in\{0,1\}$, $i=1,2$ and $n\in\mathbb Z_{\geq 0}$.
Since $W$ is an infinite group, the set of translates $\Delta^{{\rm re}}=W\Pi$ of $\Pi$ by
$W$ is infinite, and the corresponding Lie algebra $\mathfrak{g}(A)$ is infinite
dimensional.

The Weyl group translates of the
simple roots $\alpha_1$ and $\alpha_2$ give rise to a sequence of
polynomials in $a$ and $b$, where $-a$ and $-b$ are the off diagonal
entries of the generalized Cartan matrix.

In \cite{ACP} and \cite{CKMS}, the authors associate `Fibonacci type' integer sequences with the Weyl group and root lattice of these algebras. These sequences are derived explicitly from the coordinates of the real root vectors in the root space

\section{The setting}
We may identify the root lattice $ \mathbb{Z}\alpha_1 \oplus \mathbb{Z}\alpha_2$ corresponding to $A$ with $\R^{1,1}$. Then any root $\beta$ is expressed as a linear combination of the basis vectors (simple roots):
\[
\beta = x\alpha_1 + y\alpha_2 \quad \text{where } x, y \in \mathbb{Z}.
\]
The authors identify $\alpha_1$ with the coordinate  $(1,0)$ and $\alpha_2$ with $(0,1)$  to denote their coefficients in this basis.

The geometry of the root space, specifically the lengths of the roots, is encoded in the quadratic form $Q(x,y)$:
\[
Q(x,y) = 2x^2 - 2axy + \left(\frac{2a}{b}\right)y^2.
\]

 For every root $\beta\in\Delta^{{\rm re}}$ we write
$$
\beta=x(\beta)\alpha_1+y(\beta)\alpha_2,
$$
for $x(\beta), y(\beta)\in \Z$ to give an expression for $\beta$ in 
$$
\Delta^{{\rm re}}\cap(\Z\alpha_1\oplus\Z\alpha_2).
$$
Then 
$$
x(\alpha_1)=1,\ y(\alpha_1)=0
$$
$$
x(\alpha_2)=0,\ y(\alpha_2)=1.
$$
For $i=1,2$, we write $W^+\{\alpha_i\}=W\{\alpha_i\}\cap\Delta^{{\rm re}}_+$, where
$\Delta^{{\rm re}}$ denotes the real roots. Then
$$
W^+\{\alpha_1\}=\{(w_1w_2)^n\alpha_1, \ n\geq 0\}\ \sqcup\
\{w_2(w_1w_2)^n\alpha_1, \ n\geq 0\},
$$
$$
W^+\{\alpha_2\}=\{(w_2w_1)^n\alpha_2, \ n\geq 0\}\ \sqcup\
\{w_1(w_2w_1)^n\alpha_2, \ n\geq 0\}.
$$
Note that for each real root $\alpha\in\Delta^{{\rm re}}$, if we write
$\alpha=w\alpha_i$, the word $w$ is reduced and ends in $w_{3-i}$.

\medskip
\noindent
We tabulate $x(\alpha)$ and $y(\alpha)$ for all $\alpha\in\Delta^{{\rm re}}$, where
$\ell(w)=n$ for some $n\geq 0$, and $\ell(\cdot)$ is the standard length
function on $W$:
\begin{center}
\begin{tabular}{| r | r | r | r | }
\hline
&{} & {} & {} \\

$n=\ell(w)$
& $\alpha\in\{(w_1w_2)^{\lfloor n/2 \rfloor}\alpha_1\mid\ n\geq 0\}$
& $x(\alpha)$
& $y(\alpha)$
\\
\hline
$0$
& $\alpha_1$
& $1$
& $0$
\\

$2$
& $w_1w_2\alpha_1$
& $ab-1$
& $b$
\\

$4$
& $w_1w_2w_1w_2\alpha_1$
& $a^2b^2-3ab+1$
& $ab^2-2b$
\\
$\vdots$
& $\vdots$
& $\vdots$
& $\vdots$\\

\hline
\end{tabular}

\bigskip

\begin{tabular}{| r | r | r | r | }
\hline
&{} & {} & {} \\

$n=\ell(w)$
& $\alpha\in\{w_2(w_1w_2)^{\lfloor n/2 \rfloor}\alpha_1\mid \ n\geq 0\}$
& $x(\alpha)$
& $y(\alpha)$
\\
\hline
$1$
& $w_2\alpha_1$
& $1$
& $b$
\\
$3$
& $w_2w_1w_2\alpha_1$
& $ab-1$
& $ab^2-2b$

\\
$5$
& $w_2(w_1w_2)^2\alpha_1$
& $a^2b^2-3ab+1$
& $a^2b^3-4ab^2+3b$

\\

$\vdots$
& $\vdots$
& $\vdots$
& $\vdots$\\

\hline
\end{tabular}

\bigskip

\begin{tabular}{| r | r | r | r | }
\hline
&{} & {} & {} \\

$n=\ell(w)$
& $\alpha\in\{(w_2w_1)^{\lfloor n/2 \rfloor}\alpha_2\mid \ n\geq 0\}$
& $x(\alpha)$
& $y(\alpha)$
\\
\hline

$0$
& $\alpha_2$
& $0$
& $1$
\\

$2$
& $w_2w_1\alpha_2$
& $a$
& $ab-1$
\\

$4$
& $w_2w_1w_2w_1\alpha_2$
& $a^2b-2a$
& $a^2b^2-3ab+1$
\\

$\vdots$
& $\vdots$
& $\vdots$
& $\vdots$\\

\hline
\end{tabular}

\bigskip

\begin{tabular}{| r | r | r | r | }
\hline

&{} & {} & {} \\

$n=\ell(w)$
& $\alpha\in\{w_1(w_2w_1)^{\lfloor n/2 \rfloor}\alpha_2\mid \ n\geq 0\}$
& $x(\alpha)$
& $y(\alpha)$
\\
\hline
$1$
& $w_1\alpha_2$
& $a$
& $1$
\\

$3$
& $w_1w_2w_1\alpha_2$
& $a^2b-2a$
& $ab-1$

\\
$5$
& $w_1(w_2w_1)^2\alpha_2$
& $a^3b^2-4a^2b+3a$
& $a^2b^2-3ab+1$

\\

$\vdots$
& $\vdots$
& $\vdots$
& $\vdots$\\

\hline
\end{tabular}

\end{center}

\medskip
\noindent
We now rewrite the coefficients $x(\alpha)$ and $y(\alpha)$ as
tabulated above in the language of sequences. For $i=1,2$ and for
$\alpha\in W^+\{\alpha_i\}$, we have $\alpha=w\alpha_i$ for some $w\in
W$, where $\ell(w)=n$ for some $n\geq 0$, and $\ell(\cdot)$ is the standard
length function on $W$.
For $i=1,2$, we introduce the following notation for the polynomials
$x(\alpha)$ and $y(\alpha)$. We set
$$
x_{i,n}:=x_{i,n}(a,b)=x(\alpha){\text { where }}\alpha=w\alpha_i,\
\ell(w)=n,
$$
$$
y_{i,n}:=y_{i,n}(a,b)=y(\alpha){\text { where }}\alpha=w\alpha_i,\
\ell(w)=n.
$$

Once a real root $\alpha=w\alpha_i$ is fixed, the pair $(x(\alpha),y(\alpha))$ is unique for each $n=\ell(w)$.
\medskip
\begin{lemma} \label{lem1}
We have
$$
x_{1,n}=x_{1, n+1},\ n\text{ even, } n\geq 0,
$$
$$
y_{1,n}=y_{1, n-1},\ n\text{ even, } n\geq 2.
$$
\par
$$
x_{2,n}=x_{2,n-1},\ n\text{ even, } n\geq 2,
$$
$$
y_{2,n}=y_{2,n+1},\ n\text{ even, } n\geq 0.
$$
\end{lemma}

It is known that:
\begin{enumerate}
\item If $a=b=-2$,  then $A$ is affine and \cite{K}:
$$Q(\alpha,\beta) = \{j\alpha+(j+1)\beta,\  (j+1)\alpha+j\beta\ j\in \mathbb{Z}_{>0} \}.$$
\item If $a=b=-3$, then $A$ is hyperbolic and \cite{F}:
$$Q(\alpha,\beta) = \{\phi_{2j}\alpha+\phi_{(2j+2)}\beta,\  \phi_{(2j+2)}\alpha+\phi_{2j}\beta\ j\in \mathbb{Z}_{>0} \},$$ where $\phi_{j}$ is the $j^{\mbox{th}}$ Fibonacci number. 
\end{enumerate}

\section{Integer sequences in the root lattice}

For $m \geq 0$ define
\[
\begin{aligned}
X_1(a,b) &= (x_{1,2m}), & \quad Y_1(a,b) &= (y_{1,2m}), \\
X_2(a,b) &= (x_{2,2m}), & \quad Y_2(a,b) &= (y_{2,2m}).
\end{aligned}
\]

These satisfy linear recurrences such as
\[
x_{1,2m} = (ab-2)x_{1,2m-2} - x_{1,2m-4}, \qquad m \geq 2,
\]

$$
x_{2,2m} = (ab-2)x_{2,2m-2} - x_{2,2m-4}
$$

with analogous formulas for $y_{1,2m}$, $y_{2,2m}$.  
The initial conditions come from the simple roots:
\[
x_{1,0}=1,\; y_{1,0}=0, \qquad x_{2,0}=0,\; y_{2,0}=1.
\]

It follows that for each $(a,b)$ with $ab\geq 4$ we have  integer sequences:
\[
X_1(a,b) = Y_2(a,b), \qquad X_2(a,b) = Y_1(b,a).
\]
Now let $X_2(a,a) = \{u_0, u_1, u_2, \dots\}$ and $X_1(a,a) = \{v_0, v_1, v_2, \dots\}$. Define a  sequence $X(a) = \{x_0, x_1, x_2, \dots\}$ where the general term $x_n$ is given by:

\begin{align*}
x_n = 
\begin{cases} 
u_{n/2} & \text{if } n \text{ is even} \quad (\text{the } \frac{n}{2}\text{-th term of } X_2(a,a)) \\
v_{(n-1)/2} & \text{if } n \text{ is odd} \quad (\text{the } \frac{n-1}{2}\text{-th term of } X_1(a,a))
\end{cases}
\end{align*}

The sequence $X(a) = \{x_n\}_{n \ge 0}$ satisfies the following  recurrence relation:

$$
x_n = a x_{n-1} - x_{n-2} \quad \text{for } n \ge 2
$$
with initial conditions
$x_0 = 0$, 
$x_1 = 1$.

From \cite{ACP} we have

\begin{enumerate}
\item \begin{align*}X_1(a,b)&=(x_{1,2m})_{m\in \Z_{\geq 0}}=\{1, ab-1, a^2b^2-3ab+1,\dots\}\\
X_2(a,b)&=(x_{2,2m})_{m\in \Z_{\geq 0}}=\{0, a, a^2b-2a, a^3b^2-4a^2b+3a,\dots\}\end{align*}
\item $X(a) 
 = \{0, 1, a, a^2-1, a^3-2a, a^4-3a^2+1,\dots\}$.
 
\item For each pair $a,b$, $ab\ge4$, we have 
$X_1(a,b)=X_1(b,a)$.
\item For each pair $a,b$, $ab\ge4$, we have $aX_2(1,ab)=X_2(a,b).$ 
\item For each $m\in \Z_{\geq0}$, $a,b\in\Z_{\geq 1}$, $ab\geq 4$ we have
$$
x_{1,2m}(a,b)\quad=\quad
U_{2m}\left(\frac{\sqrt{ab}}{2}\right ),
$$
and
$$
x_{2,2m-1}(a,b)\quad=\quad \sqrt{\frac{a}{b}}\cdot
U_{2m-1}\left(\frac{\sqrt{ab}}{2}\right ),
$$
where $U_{n}\left(\frac{\sqrt{ab}}{2}\right )$ is the Chebyshev
$U$-polynomial $U_n(x)$ evaluated at $x=\frac{\sqrt{ab}}{2}$.
\item For each $i=1, 2$ and $a\geq 2$, $X_i(a,a)$ and
$Y_i(a,a)$ are  bisections of $X(a)$.
\end{enumerate}

\newpage
Recall that  $X(k) = \{x_0, x_1, x_2, \dots\}$ is defined by:
\begin{align*}
    x_0 &= 0, \quad x_1 = 1 \\
    x_n &= k x_{n-1} - x_{n-2} \quad \text{for } n \ge 2
\end{align*}

\begin{theorem} The sequences $X_2(a,b)$ and $X_1(a,b)$ can be expressed in terms of the  sequence $X(k)$, where $k = ab-2$ as follows:
\begin{enumerate}
 
\item 
    \[ X_2(a,b) = a \cdot X(ab-2) \]
    
    \item 
    \[ X_1(a,b)_n = X(ab-2)_{n} + X(ab-2)_{n+1} \]
\end{enumerate}
\end{theorem}
\begin{proof}
For (1), let $Y_n$ be the $n$-th term of $X_2(a,b)$ and $Z_n$ be the $n$-th term of $a \cdot X(k)$.

    The sequence $X_2(a,b)$ satisfies the recurrence $Y_n = (ab-2)Y_{n-1} - Y_{n-2}$.
    The sequence $X(k)$ satisfies $x_n = k x_{n-1} - x_{n-2}$. So $Z_n$ satisfies $Z_n = k Z_{n-1} - Z_{n-2}$.
    Since $k = ab-2$, both satisfy the {same recurrence}.

  For $X_2(a,b)$: $Y_0 = 0$, $Y_1 = a$.
      For $a \cdot X(k)$: $Z_0 = a(0) = 0$, $Z_1 = a(1) = a$.

 Since the recurrence and initial values are identical, the sequences are identical.

For (2), let $S_n = x_n + x_{n+1}$ be the sequence of sums from $X(k)$.

    We verify that $S_n$ satisfies the standard recurrence $S_n = k S_{n-1} - S_{n-2}$.
    \begin{align*}
        k S_{n-1} - S_{n-2} &= k(x_{n-1} + x_n) - (x_{n-2} + x_{n-1}) \\
        &= (k x_{n-1} - x_{n-2}) + (k x_n - x_{n-1}) \\
        &= x_n + x_{n+1} \\
        &= S_n
    \end{align*}

The sequence $X_1(a,b)$ is defined by the initial conditions $x_{1,0} = 1$ and $x_{1,2} = ab-1$.
  We have $S_0 = x_0 + x_1 = 0 + 1 = 1$, which equals the first term of $X_1(a,b)$.
    And $S_1 = x_1 + x_2 = 1 + (k \cdot 1 - 0) = 1 + k$.
  Setting $k = ab-2$ we obtain
    \[ S_1 = 1 + (ab-2) = ab - 1 \]
  which matches the second term of $X_1(a,b)$.

\end{proof}

\clearpage
\section{Tables of integer sequences}
\begin{table}[h!]
\centering
\caption{$X(j)$, for $j=2,\dots,20$}
\label{tab:sequence_data_x_j}
\begin{tabular}{|c|p{6cm}|p{7.5cm}|}
\hline 
 \textbf{j} & \textbf{Sequence Terms} & \textbf{OEIS entry and description} \\ \hline
$2$ & 0, 1, 2, 3, 4, \dots &  A001477 \\ \hline
$3$ & 0, 1, 3, 8, 21, \dots & A001906 Bisection of the Fibonacci sequence \\ \hline
$4$ & 0, 1, 4, 15, 56, \dots & A001353 The number of spanning trees in a $2\times n$ grid
\\ \hline
$5$ & 0, 1, 5, 24, 115, \dots & A004254 \\ \hline
$6$ & 0, 1, 6, 35, 204, \dots & A001109 Square roots of square triangular numbers  \\ \hline
$7$ & 0, 1, 7, 48, 329, \dots & A004187 \\ \hline
$8$ & 0, 1, 8, 63, 496, \dots & A001090 \\ \hline
$9$ & 0, 1, 9, 80, 711, \dots & A018913 \\ \hline
$10$ & 0, 1, 10, 99, 980, \dots & A004189  Indices of square numbers which are also generalized pentagonal numbers \\ \hline
$11$ & 0, 1, 11, 120, 1309, \dots & A004190 \\ \hline
$12$ & 0, 1, 12, 143, 1704, \dots & A004191 \\ \hline
$13$ & 0, 1, 13, 168, 2171, \dots & A078362 \\ \hline
$14$ & 0, 1, 14, 195, 2716, \dots & A007655 \\ \hline
$15$ & 0, 1, 15, 224, 3345, \dots & A078364 \\ \hline
$16$ & 0, 1, 16, 255, 4064, \dots & A077412 \\ \hline
$17$ & 0, 1, 17, 288, 4879, \dots & A078366 \\ \hline
$18$ & 0, 1, 18, 323, 5796, \dots & A049660 \\ \hline
$19$ & 0, 1, 19, 360, 6821, \dots & A078368 \\ \hline
$20$ & 0, 1, 20, 399, 7960, \dots & A075843 \\ \hline
\end{tabular}
\end{table}

 \begin{figure}[hbt!]\begin{center}
\includegraphics[scale=0.25]{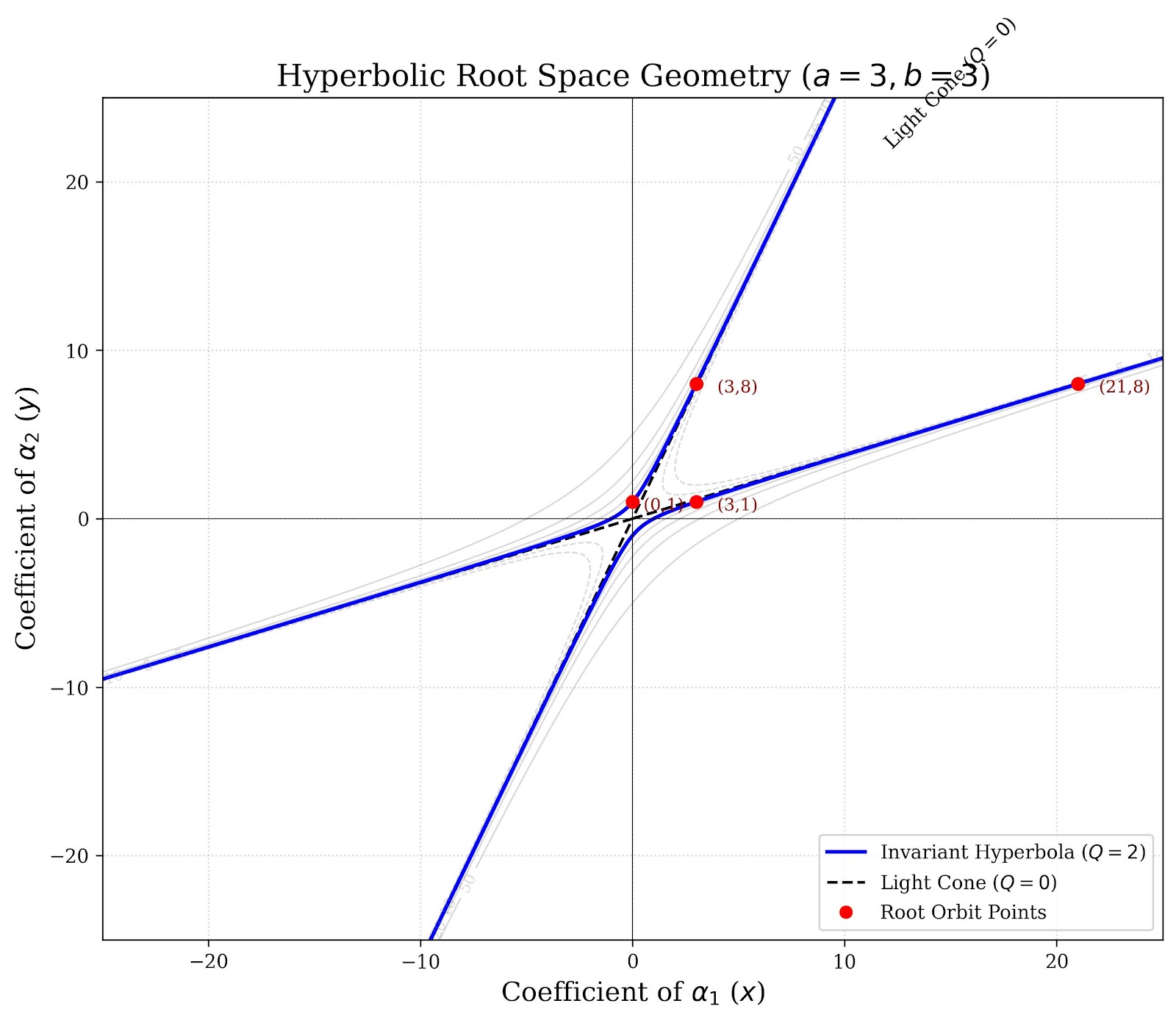}
\end{center}\end{figure}

\clearpage
\begin{table}[h!]
\centering
\caption{$X_1(1,j)$ for $j=4,\dots,20$: Descriptions and OEIS References}
\label{tab:sequence_data_desc}
\begin{tabular}{| r | l |  l  |}
\hline
\textbf{j} & \textbf{Sequence terms} & \textbf{OEIS entry and description}\\
\hline

$4$ & $ 1,3,5,7,9,\dots$ & \text{ A005408}\\
\hline

$5$ & $ 1,4,11,29,76,\dots$ & 

 A002878: Bisection of the Lucas sequence
\\

\hline

$6$ & $ 1,5,19,71,265,\dots$ & \text{ A001834}\\
\hline

$7$ & $ 1,6,29,139,666,\dots$ & 
$\begin{matrix}
{\text{ A030221:}}
{\text{ Chebyshev even indexed $U$-polynomials}}\\ 
{\text{ evaluated at $\sqrt{7}/2$}}
\end{matrix}$
\\

\hline

$8$ & $ 1,7,41,239,1393,\dots$ & 
$\begin{matrix}
&{\text{ A002315:}}
&{\text{ the Newman-Shanks-Williams numbers}}
\end{matrix}$
\\

\hline

$9$ & $ 1,8,55,377,2584,\dots$ &

A033890: Fibonacci$(4n+2)$
\\

\hline
$10$ & $ 1,9,71,559,4401,\dots$ & 
$\begin{matrix}
{\text{ A057080:}}
{\text{ Chebyshev even indexed $U$-polynomials}}
{\text{ evaluated }}\\
{\text{at $\sqrt{10}/2$ with $\lim_{n\longrightarrow\infty} a(n)/a(n-1) = 4 + \sqrt{15}$}}
\end{matrix}$
\\
\hline
$11$ & $ 1,10,89,791,7030,\dots$ & 
$\begin{matrix}
&{\text{ A057081:}}
{\text{ Chebyshev even indexed $U$-polynomials}}\\ 
&{\text{ evaluated at $\sqrt{11}/2$}}
\end{matrix}$
\\

\hline
$12$ & $ 1,11,109,1079,10681,\dots$ & 
$\begin{matrix}
{\text{ A054320:}}
{\text{ Chebyshev even indexed $U$-polynomials}}
{\text{ evaluated}}\\ 
{\text{ at $\sqrt{3}$. Squares of entries are star numbers}}
\end{matrix}$
\\
\hline
$13$ & $  1,12,131,1429,15588, \dots$ & 
$\begin{matrix}
{\text{ A097783:}}
{\text{ Chebyshev polynomials}}
{\text{ with diophantine property}}
\end{matrix}$ 
\\
\hline

$14$ & $  1,13,155,1847,22009, \dots$ & 
$\begin{matrix}
{\text{ A077416:}}
{\text{ Chebyshev $S$-sequence}}
{\text{with diophantine property}}
\end{matrix}$ 
\\
\hline

$15$ & $ 1, 14, 181, 2339, 30226,\dots$ & \text{ A126866}\\
\hline
$16$ & $  1,15,209,2911,40545, \dots$ & 
$\begin{matrix}
{\text{ A028230:}}
{\text{ Indices of square numbers which}}
{\text{ are also octagonal}}
\end{matrix}$
\\
\hline
$17$ & $ 1, 16, 239, 3569, 53296,\dots$ &
$\begin{matrix}
{\text{ A161591:}}
{\text{ The list of the B values in the common solutions}}\\
{\text{ to the $2$ equations $13k+1 = A^2$, $17k+1 = B^2$}}
\end{matrix}$
\\
\hline
$18$ & $ 1, 17, 271, 4319, 68833,\dots$ &
$\begin{matrix}
{\text{ A159678:}}
{\text{
 The $b(j)$ solutions of the 2-equation problem }}\\
{ \text{$7n(j) + 1 = a(j)*a(j)$ and $9n(j) + 1 = b(j)*b(j)$}}\\
 {\text{ with positive integer numbers.}}

\end{matrix}$

\\
\hline
$19$ & $ 1, 18, 305, 5167, 87534,\dots$ &
$\begin{matrix}
{\text{ A161599:}}
{\text{ The list of the B values in the common solutions}}\\
{\text{to the $2$ equations $15k+1 = A^2$, $19k+1 = B^2$}}
\end{matrix}$
\\
\hline
$20$ & $ 1,19,341,6119,109801,\dots$ & 
$\begin{matrix}
{\text{ A049629:}}
{\text{ $a(n)=1/4[F(6n+4)+F(6n+2)]$}}\\
{\text{ where $F(n)$ is the Fibonacci sequence}}
\end{matrix}$
\\
\hline
\end{tabular}
\end{table}

\clearpage
\begin{table}[h!]
\centering
\caption{$X_1(2,j)$ for $j=1,\dots,20$}
\vspace{0.2cm}
\begin{tabular}{|c|l|p{7cm}|}
\hline
\textbf{j} & \textbf{Sequence terms}  & \textbf{OEIS entry and Description} \\ \hline
1 & 1, 1, -1, -1, 1, 1, -1, \dots & {A057078}: Periodic sequence (Period 6). \\ \hline
2 & 1, 3, 5, 7, 9, 11, 13, \dots & {A005408}: The odd numbers. \\ \hline
3 & 1, 5, 19, 71, 265, 989, \dots & {A001834}: $a(n) = 4a(n-1) - a(n-2)$. \\ \hline
4 & 1, 7, 41, 239, 1393, 8119, \dots & {A002315}: Newman-Shanks-Williams numbers. \\ \hline
5 & 1, 9, 71, 559, 4401, 34649, \dots & {A057080}: Values of $U_{2n}(\frac{\sqrt{10}}{2})$. \\ \hline
6 & 1, 11, 109, 1079, 10681, \dots & {A054320}: Values of $U_{2n}(\sqrt{3})$. \\ \hline
7 & 1, 13, 155, 1847, 22009, \dots & {A077416}: Chebyshev S-sequence. \\ \hline
8 & 1, 15, 209, 2911, 40545, \dots &  \\ \hline
9 & 1, 17, 271, 4319, 68833, \dots & {A126866}: Recurrence $a(n) = 16a(n-1) - a(n-2)$. \\ \hline
10 & 1, 19, 341, 6119, 109801, \dots & {A049629}: Recurrence $a(n) = 18a(n-1) - a(n-2)$. \\ \hline
11 & 1, 21, 419, 8359, 166761, \dots & {A083043}: Recurrence $a(n) = 20a(n-1) - a(n-2)$. \\ \hline
12 & 1, 23, 505, 11087, 243409, \dots & {A133283}: Numbers $k$ such that $30k^2 + 6$ is a square \\ \hline
13 & 1, 25, 599, 14351, 343825, \dots & {A159661}: The $b(j)$ solutions to the $2$ equations problem $11n(j) + 1 = a(j)^2$ and $13n(j)+1 = b(j)^2$ with positive integer elements.\\ \hline
14 & 1, 27, 701, 18199, 472473, \dots &  {A157461}: Expansion of $\frac{x(x+1)}{x^2-26x+1}.$\\ \hline
15 & 1, 29, 811, 22679, 634201, \dots &  {A159669}: Expansion of $\frac{x(x+1)}{x^2-28x+1}.$\\ \hline
16 & 1, 31, 929, 27839, 834241, \dots &  {A157878}: Expansion of $\frac{x(x+1)}{x^2-30x+1}.$ \\ \hline
17 & 1, 33, 1055, 33727, 1078209, \dots & {A159675}: Expansion of $\frac{x(x+1)}{x^2-32x+1}.$  \\ \hline
18 & 1, 35, 1189, 40391, 1372105, \dots & {A046176}: Indices of hexagonal square numbers. \\ \hline
19 & 1, 37, 1331, 47879, 1722313, \dots &  \\ \hline
20 & 1, 39, 1481, 56239, 2135601, \dots & {A097314}: Pell equation solutions related to $38$. \\ \hline
\end{tabular}
\end{table}

\clearpage
\begin{table}[h!]
\centering
\caption{$X_1(3,j)$ for $j=1,\dots,20$}
\vspace{0.2cm}
\begin{tabular}{|c|l|l|}
\hline
\textbf{j} & \textbf{Sequence terms} & \textbf{OEIS entry and description} \\ \hline
1 & 1, 2, 1, -1, -2, -1, 1, \dots & A057078 \\ \hline
2 & 1, 5, 19, 71, 265, 989, \dots & A001834 \\ \hline
3 & 1, 8, 55, 377, 2584, 17711, \dots & A033890 \\ \hline
4 & 1, 11, 109, 1079, 10681, \dots & A054320 \\ \hline
5 & 1, 14, 181, 2339, 30226, \dots & A126866 \\ \hline
6 & 1, 17, 271, 4319, 68833, \dots & A028230 \\ \hline
7 & 1, 20, 379, 7181, 136060, \dots &  \\ \hline
8 & 1, 23, 505, 11087, 243409, \dots & A133283  \\ \hline
9 & 1, 26, 649, 16199, 404326, \dots &  \\ \hline
10 & 1, 29, 811, 22679, 634201, \dots & A159669 \\ \hline
11 & 1, 32, 991, 30689, 950368, \dots &  \\ \hline
12 & 1, 35, 1189, 40391, 1372105, \dots & A046176 \\ \hline
13 & 1, 38, 1405, 51947, 1920634, \dots &  \\ \hline
14 & 1, 41, 1639, 65519, 2619121, \dots &  \\ \hline
15 & 1, 44, 1891, 81269, 3492676, \dots &  \\ \hline
16 & 1, 47, 2161, 99359, 4568353, \dots & A189173 \\ \hline
17 & 1, 50, 2449, 119951, 5875150, \dots & $a_n + a_{n+1} = \frac{51}{7}$*A154024 \\ \hline
18 & 1, 53, 2755, 143207, 7444009, \dots & $\frac{1}{5}$ * A175180 \\ \hline
19 & 1, 56, 3079, 169289, 9307816, \dots &  \\ \hline
20 & 1, 59, 3421, 198359, 11501401, \dots &  \\ \hline
\end{tabular}
\end{table}

\clearpage

\begin{table}[h!]
\centering
\caption{$X_1(4,j)$ for $j=1,\dots,20$}
\vspace{0.2cm}
\begin{tabular}{|c|l|l|}
\hline
\textbf{j} & \textbf{Sequence terms}  & \textbf{OEIS entry and description} \\ \hline
1 & 1, 3, 5, 7, 9, 11, 13, \dots & A005408 \\ \hline
2 & 1, 7, 41, 239, 1393, 8119, \dots & A002315 \\ \hline
3 & 1, 11, 109, 1079, 10681, 105731, \dots & A054320 \\ \hline
4 & 1, 15, 209, 2911, 40545, 564719, \dots & A077416 \\ \hline
5 & 1, 19, 341, 6119, 109801, 1970300, \dots & A049629 \\ \hline
6 & 1, 23, 505, 11087, 243409, 5343911, \dots & A133283 \\ \hline
7 & 1, 27, 701, 18199, 472473, 12266099, \dots & A157461 \\ \hline
8 & 1, 31, 929, 27839, 834241, 25000391, \dots & A157878 \\ \hline
9 & 1, 35, 1189, 40391, 1372105, 46611179, \dots & A046176 \\ \hline
10 & 1, 39, 1481, 56239, 2135601, 81096599, \dots & A097314 \\ \hline
11 & 1, 43, 1805, 75767, 3180409, 133501399, \dots &  \\ \hline
12 & 1, 47, 2161, 99359, 4568353, 210044879, \dots & A189173 \\ \hline
13 & 1, 51, 2549, 127399, 6367401, 318242651, \dots & $a_{n} + a_{n+1} =$ A004296 \\ \hline
14 & 1, 55, 2969, 160271, 8651665, 466902964, \dots &  \\ \hline
15 & 1, 59, 3421, 198359, 11501401, 666882899, \dots &  \\ \hline
16 & 1, 63, 3905, 242047, 15003009, 929944511, \dots & A001090 (bisection) \\ \hline
17 & 1, 67, 4421, 291719, 19249033, 1270144459, \dots & A078989 \\ \hline
18 & 1, 71, 4969, 347759, 24338161, 1703323511, \dots &  \\ \hline
19 & 1, 75, 5549, 410551, 30375225, 2247356099, \dots &  \\ \hline
20 & 1, 79, 6161, 480499, 37472761, 2922394859, \dots &  \\ \hline
\end{tabular}
\end{table}

\clearpage
\begin{table}[h!]
\centering
\caption{$X_2(1,j)$ for $j=1,\dots,20$}
\vspace{0.2cm}
\begin{tabular}{|c|l|l|}
\hline
\textbf{j} & \textbf{Sequence terms}  & \textbf{OEIS entry and Description} \\ \hline
1 & 1, 1, -1, -1, 1, 1, -1, \dots & Periodic sequence (Period 6) \\ \hline
2 & 1, 3, 5, 7, 9, 11, 13, \dots & {A005408}: The odd numbers \\ \hline
3 & 1, 5, 19, 71, 265, 989, \dots & {A001834}: $a(n) = 4a(n-1) - a(n-2)$ \\ \hline
4 & 1, 7, 41, 239, 1393, 8119, \dots & {A002315}: NSW numbers \\ \hline
5 & 1, 9, 71, 559, 4401, 34649, \dots & {A057080}: $U_{2n}(\sqrt{10}/2)$ \\ \hline
6 & 1, 11, 109, 1079, 10681, \dots & {A054320}: $U_{2n}(\sqrt{3})$ \\ \hline
7 & 1, 13, 155, 1847, 22009, \dots & {A077416}: Chebyshev $S$-sequence with Diophantine property \\ \hline
8 & 1, 15, 209, 2911, 40545, \dots & {A028230}: Indices of octagonal square numbers \\ \hline
9 & 1, 17, 271, 4319, 68833, \dots & {A126866}: $a(n) = 16a(n-1) - a(n-2)$ \\ \hline
10 & 1, 19, 341, 6119, 109801, \dots & {A049629}: $a(n) = 18a(n-1) - a(n-2)$ \\ \hline
11 & 1, 21, 419, 8359, 166761, \dots & {A083043}: $a(n) = 20a(n-1) - a(n-2)$ \\ \hline
12 & 1, 23, 505, 11087, 243409, \dots & {A133283}: Numbers $k$ such that $30k^2 + 6$ is a square \\ \hline
13 & 1, 25, 599, 14351, 343825, \dots & {A159661}: Recurrence with $24a(n-1)$ \\ \hline
14 & 1, 27, 701, 18199, 472473, \dots &  {A157461}: Expansion of $\frac{x(x+1)}{x^2-26x+1}.$\\ \hline
15 & 1, 29, 811, 22679, 634201, \dots &  {A159669}: Expansion of $\frac{x(x+1)}{x^2-28x+1}.$\\ \hline
16 & 1, 31, 929, 27839, 834241, \dots &  {A157878}: Expansion of $\frac{x(x+1)}{x^2-30x+1}.$ \\ \hline
17 & 1, 33, 1055, 33727, 1078209, \dots & {A159675}: Expansion of $\frac{x(x+1)}{x^2-32x+1}.$  \\ \hline
18 & 1, 35, 1189, 40391, 1372105, \dots & {A046176}: Indices of hexagonal square numbers \\ \hline
19 & 1, 37, 1331, 47879, 1722313, \dots &  \\ \hline
20 & 1, 39, 1481, 56239, 2135601, \dots & {A097314}: Pell equation solutions related to $38$. \\ \hline
\end{tabular}
\end{table}

\clearpage

\clearpage\begin{table}[h!]
\centering
\caption{$X_2(2,j)$ for $j=2,\dots,20$}
\vspace{0.2cm}
\begin{tabular}{|c|l|l|}
\hline
\textbf{j} & \textbf{Sequence terms}  & \textbf{OEIS entry and description} \\ \hline
2 & 2, 4, 6, 8, 10, 12, 14, 16, \dots & {A005843}: The even numbers \\ \hline
3 & 2, 8, 30, 112, 418, 1560, \dots & {A005319} \\ \hline
4 & 2, 12, 70, 408, 2378, 13860, \dots & {A054486}: Even bisection of A001541 \\ \hline
5 & 2, 16, 126, 992, 7810, 61488, \dots & {A057076} \\ \hline
6 & 2, 20, 198, 1960, 19402, 192060, \dots & {A057077} \\ \hline
7 & 2, 24, 286, 3408, 40610, 483912, \dots & {A057079} \\ \hline
8 & 2, 28, 390, 5432, 75658, 1053780, \dots & {A157065} \\ \hline
9 & 2, 32, 510, 8128, 129538, 2064480, \dots & {A157456} \\ \hline
10 & 2, 36, 646, 11592, 208010, 3732588, \dots & {A157648} \\ \hline
11 & 2, 40, 798, 15920, 317602, 6336120, \dots & {A157778} \\ \hline
12 & 2, 44, 966, 21208, 465610, 10222212, \dots & {A158228} \\ \hline
13 & 2, 48, 1150, 27552, 660098, 15814800, \dots & {A158586} \\ \hline
14 & 2, 52, 1350, 35048, 909898, 23622300, \dots & {A158679} \\ \hline
15 & 2, 56, 1566, 43792, 1224610, 34245288, \dots & {A158768} \\ \hline
16 & 2, 60, 1798, 53880, 1614602, 48384180, \dots & {A158819} \\ \hline
17 & 2, 64, 2046, 65408, 2091010, 66846912, \dots & {A158849} \\ \hline
18 & 2, 68, 2310, 78472, 2665738, 90556620, \dots & {A158888} \\ \hline
19 & 2, 72, 2590, 93168, 3351458, 120559320, \dots & {A158913} \\ \hline
20 & 2, 76, 2886, 109592, 4161610, 158031588, \dots & {A158941} \\ \hline
\end{tabular}
\end{table}
\renewcommand{\refname}{}
\clearpage

\begin{table}[h!]
\centering
\caption{$X_2(3,j)$ for $j=1,\dots,20$}
\vspace{0.2cm}
\begin{tabular}{|c|l|l|}
\hline
\textbf{j} & \textbf{Sequence terms} & \textbf{OEIS entry and description} \\ \hline
1 & 3, 3, 0, -3, -3, 0, 3, \dots & A084103\\ \hline
2 & 3, 12, 45, 168, 627, 2340, \dots & A005320 \\ \hline
3 & 3, 21, 144, 987, 6765, \dots & A014445 \\ \hline
4 & 3, 30, 297, 2940, 29103, \dots & $3$A004189 \\ \hline
5 & 3, 39, 504, 6513, 84165, \dots & $3$A078362 \\ \hline
6 & 3, 48, 765, 12192, 194307, \dots & A001080  \\ \hline
7 & 3, 57, 1080, 20463, 387717, \dots & $3$A078368 \\ \hline
8 & 3, 66, 1449, 31812, 698415, \dots & $3$A077421 \\ \hline
9 & 3, 75, 1872, 46725, 1166253, \dots & $3$A097780\\ \hline
10 & 3, 84, 2349, 65688, 1836915, \dots & $3$A097311 \\ \hline
11 & 3, 93, 2880, 89187, 2761917, \dots &  $3$A200442\\ \hline
12 & 3, 102, 3465, 117708, 3998607, \dots & $3$A029547 \\ \hline
13 & 3, 111, 4104, 151737, 5610165, \dots & $3$A206565 \\ \hline
14 & 3, 120, 4797, 191760, 7665603, \dots &  \\ \hline
15 & 3, 129, 5544, 238263, 10239765, \dots &  \\ \hline
16 & 3, 138, 6345, 291732, 13413327, \dots & $\frac{3}{4}$ A174772\\ \hline
17 & 3, 147, 7200, 352653, 17272797, \dots &  $\frac{3}{7}$ A154024\\ \hline
18 & 3, 156, 8109, 421512, 21910515, \dots & $3$A370188 \\ \hline
19 & 3, 165, 9072, 498795, 27424653, \dots &  \\ \hline
20 & 3, 174, 10089, 584988, 33919215, \dots & $\frac{1}{20}$A004297 \\ \hline
\end{tabular}
\end{table}

\clearpage

\begin{table}[h!]
\centering
\caption{$X_2(4,j)$ for $j=1,\dots,20$}
\vspace{0.2cm}
\begin{tabular}{|c|l|l|}
\hline
\textbf{j} & \textbf{Sequence terms}  & \textbf{OEIS entry and description} \\ \hline
1 & 4, 8, 12, 16, 20, 24, \dots & A008586 \\ \hline
2 & 4, 24, 140, 816, 4756, 27720, \dots & A005319 \\ \hline
3 & 4, 40, 396, 3920, 38804, 384120, \dots & A122652 \\ \hline
4 & 4, 56, 780, 10864, 151316, 2107560, \dots & 4*A007655'  \\ \hline
5 & 4, 72, 1292, 23184, 416020, 7465176, \dots & A060645 \\ \hline
6 & 4, 88, 1932, 42416, 931220, 20444424, \dots & 4*A077421 \\ \hline
7 & 4, 104, 2700, 70096, 1819796, 47244600, \dots & 2*A174779 \\ \hline
8 & 4, 120, 3596, 107760, 3229204, 96768360, \dots & A068204 \\ \hline
9 & 4, 136, 4620, 156944, 5331476, 181113240, \dots & A202299' \\ \hline
10 & 4, 152, 5772, 219184, 8323220, 316063176, \dots & 4 * A078987  \\ \hline
11 & 4, 168, 7052, 296016, 12425620, 521580024, \dots & $\frac{1}{11}$ * A004295' \\ \hline
12 & 4, 184, 8460, 388976, 17884436, 822295080, \dots & A174772' \\ \hline
13 & 4, 200, 9996, 499600, 24970004, 1248000600, \dots & A174776' \\ \hline
14 & 4, 216, 11660, 629424, 33977236, 1834141320, \dots &  \\ \hline
15 & 4, 232, 13452, 779984, 45225620, 2622305976, \dots & $\frac{1}{15}$ * A004297' \\ \hline
16 & 4, 248, 15372, 952816, 59059220, 3660718824, \dots &  \\ \hline
17 & 4, 264, 17420, 1149456, 75846676, 5004731160, \dots & 4 * A097316 \\ \hline
18 & 4, 280, 19596, 1371440, 95981204, 6717312840, \dots & $\frac{2}{3} $ * A174773' \\ \hline
19 & 4, 296, 21900, 1620304, 119880596, 8869543800, \dots & $\frac{2}{3}$ * A174777' \\ \hline
20 & 4, 312, 24332, 1897584, 147987220, 11541105576, \dots &  \\ \hline
\end{tabular}
\end{table}
{\bf Note:} An OEIS sequence with ' after the sequence number denotes the OEIS sequence with the first term deleted.

\end{document}